\documentclass[11pt]{amsart}

\usepackage{amsmath}
\usepackage{amsthm}
\usepackage{amssymb}
\usepackage{amsfonts}
\usepackage{latexsym}
\usepackage{mathrsfs}
\usepackage{multicol}
\usepackage{fancyvrb}

\usepackage{lmodern}
\usepackage[T1]{fontenc}

\usepackage[margin=1in]{geometry}
\usepackage{paralist}

\usepackage{color}

\usepackage[colorlinks,pagebackref=true]{hyperref}

\newcommand{\arxiv}[1]{\href{http://arxiv.org/abs/#1}{{\tt arXiv:#1}}}

\makeatletter

\@addtoreset{equation}{section}
\def\theequation{\thesection.\@arabic \c@equation}

\def\@citecolor{blue}
\def\@linkcolor{blue}
\def\@urlcolor{blue}

\def\theenumi{\@alph\c@enumi}

\makeatother

\theoremstyle{plain}
\newtheorem{theorem}[equation]{Theorem}
\newtheorem{lemma}[equation]{Lemma}

\newtheorem{proposition}[equation]{Proposition}

\theoremstyle{definition}

\newtheorem{remark}[equation]{Remark}
\newenvironment{remarkbox}[1][]{%
    \begin{remark}[#1] \pushQED{\qed}}{\popQED \end{remark}}

\newtheorem{example}[equation]{Example}

\newtheorem{definition}[equation]{Definition}
\newenvironment{definitionbox}[1][]{%
    \begin{definition}[#1]\pushQED{\qed}}{\popQED \end{definition}}

\newtheorem{notation}[equation]{Notation}

\newtheorem{discussion}[equation]{Discussion}
\newenvironment{discussionbox}[1][]{%
    \begin{discussion}[#1]\pushQED{\qed}}{\popQED \end{discussion}}

\newtheorem{construction}[equation]{Construction}

\DeclareMathOperator{\calF}{\mathcal F}

\newcommand{\frakm}{{\mathfrak m}}

 \let\strSh\calO

\newcommand{\bbV}{\mathbb{V}}

\newcommand{\ints}{\mathbb{Z}}
\newcommand{\rationals}{\mathbb{Q}}

\newcommand{\BettiCone}{\mathbb{B}}

\DeclareMathOperator{\pdim}{pdim}

\DeclareMathOperator{\rank}{rank}
\DeclareMathOperator{\projective}{\mathbb{P}}
\DeclareMathOperator{\Tor}{Tor}
\DeclareMathOperator{\homology}{H}
\newcommand{\define}[1]{{\bf #1}}
\DeclareMathOperator{\image}{image}
\DeclareMathOperator{\depth}{depth}

\begin{document}
\title{The cone of Betti tables over a rational normal curve}

\author{Manoj Kummini}
\address{Chennai Mathematical Institute, Siruseri, Tamilnadu 603103, India}
\email{mkummini@cmi.ac.in}

\author{Steven V Sam}
\address{Department of Mathematics, University of California, Berkeley,
USA}
\email{svs@math.berkeley.edu}

\date{October 29, 2013}

\thanks
{MK was partially supported by a CMI Faculty Development Grant; SVS was
supported by a Miller research fellowship. Additionally, both authors
thank Mathematical Sciences Research Institute, Berkeley, where they were
members during Fall 2012 when part of this work was done.}


\begin{abstract}
We describe the cone of Betti tables of Cohen--Macaulay modules over the
homogeneous coordinate ring of a rational normal curve. 
\end{abstract}

\maketitle

\section{Introduction}

The study of the cone generated by the graded Betti tables of finitely
generated modules over graded rings has received much attention recently.
(See Definition~\ref{definitionbox:BScone} for the relevant definitions.)
This began with a conjectural description of this cone in the case of 
polynomial rings by M.~Boij and J.~S\"oderberg~\cite{boijsoderberg} 
which was proved by D.~Eisenbud and F.-O.~Schreyer~\cite{es:bs}. 
We refer to \cite{es:survey, floystad} for a survey of this development and
related results. Similarly, in the local case, 
there is a description of the cone of Betti sequences over regular local
rings~\cite{BEKS}.

However, not much is known about the cone of Betti tables over other graded
rings, or over non-regular local rings.  The cone of Betti tables for rings
of the form $\Bbbk[x,y] / q(x,y)$ where $q$ is a homogeneous quadric is
described in \cite{bbeg}. In the local hypersurface case,~\cite{BEKS} gives
some partial results and some asymptotic results. We also point to \cite[\S\S 9, 10]{eisenbuderman} for a study of Betti tables in the non-regular case.

In this paper, we consider the coordinate ring of a rational normal curve.
These rings are of finite Cohen--Macaulay representation type, and the
syzygies of maximal Cohen--Macaulay modules have a simple description; see
Discussion~\ref{discussionbox:mcmRNC}. 
Our main result is Theorem~\ref{thm:betticoneB}, describing the cone
generated by finite-length modules over such a ring.
Remark~\ref{remarkbox:depthOne} explains how the argument
extends to Cohen--Macaulay modules of higher depth.
We work out a few explicit examples of our result in \S\ref{sec:example} for the rational normal cubic.
In Remark~\ref{remarkbox:totalBetti}, we consider the cone generated by
sequences of total Betti numbers, and get a picture reminiscent of the case
of regular local rings from~\cite{BEKS}.

We thank Jerzy Weyman for helpful discussions and Daniel Erman for suggestions which significantly improved an earlier draft.

\section{Preliminaries}

Let $\Bbbk$ be a field, which we fix for the rest of the article.

\begin{definitionbox}
\label{definitionbox:BScone}
Let $R$ be any Noetherian graded $\Bbbk$-algebra. For a finitely generated
$R$-module $M$, define its \define{graded Betti numbers} 
$\beta_{i,j}^R(M) := \dim_\Bbbk \Tor_i^R(\Bbbk,M)_j$. 
Let $t = \pdim(R) + 1$ (possibly $t = \infty$). 
The \define{Betti table} of $M$ is 
\[
\beta^R(M) := 
\left( \beta^R_{i,j}(M)\right)_{ \substack{0 \leq i < t, \\ j \in \ints}},
\]
which is an element of the $\rationals$-vector space
\[
\bbV_R := \prod_{0 \leq i < t} \bigoplus_{j \in \ints} \rationals.
\] 
The \define{cone of Betti tables} over $R$ is the cone 
$\BettiCone(R)$ generated by the rays $\rationals_{\ge 0}\cdot\beta^R(M)$
in $\bbV_R$.
\end{definitionbox}

Let $S = \Bbbk[x,y]$. Fix $d \geq 1$. Let $B
= \bigoplus_n S_{nd} \subset S$, i.e., the homogeneous coordinate ring of
the rational normal curve of degree $d$.  
For a coherent sheaf $\calF$ on 
$\projective^1$, define 
\[
\Gamma_*^{(d)}  (\calF) = \bigoplus_{j \in \ints} \homology^0(\projective^1, \calF \otimes \strSh(dj)).
\]
We set $\Gamma_* = \Gamma_*^{(1)}$.
Also, for a finitely generated $B$-module $M$, let $\widetilde{M}$ be the
associated coherent sheaf on $\projective^1$. There is an exact sequence
\[
0 \to \homology^0_{\frakm}(M) \to M \to \Gamma^{(d)}_\ast(\widetilde{M}) \to
\homology^1_{\frakm}(M) \to 0
\]
where $\homology^i_{\frakm}$ denotes local cohomology with respect to the
homogeneous maximal ideal $\frakm \subset B$ \cite[Theorem 13.21]{localcoh}
and hence the map $M \to \Gamma^{(d)}_\ast(\widetilde{M})$ is an isomorphism if
(and only if) 
$M$ is a maximal Cohen--Macaulay module by \cite[Theorem 9.1]{localcoh}.

\begin{discussionbox}[Maximal Cohen--Macaulay modules over $B$]
\label{discussionbox:mcmRNC}
Ignoring the grading for a moment, the indecomposable maximal Cohen--Macaulay $B$-modules are exactly the modules 
\[
M^{(\ell)} := \bigoplus_{n \ge 0} S_{nd+\ell}, \text{ for } \ell=0, \ldots, d-1. 
\]
To see this, let $M$ be a maximal Cohen--Macaulay $B$-module. Then
$\widetilde{M}$ is a vector bundle on $\projective^1$, and by
Grothendieck's theorem, every vector bundle on $\projective^1$ is a direct
sum of line bundles. Note that $\Gamma^{(d)}_\ast(\strSh(i)) = M^{(\ell)}$ if $i
\equiv \ell \pmod d$ and $0 \le \ell < d$. Since
$\Gamma^{(d)}_\ast(\widetilde{M}) \cong M$, we conclude that $M$ is a direct sum
of the $M^{(\ell)}$ for various $\ell$.

For each $0 \le \ell \le d-1$, consider the exact sequence
\[
0 \to \strSh_{\projective^1}(-1)^\ell \to \homology^0(\projective^1, \strSh_{\projective^1}(\ell)) \otimes \strSh_{\projective^1} \to \strSh_{\projective^1}(\ell) \to 0.
\]
Applying $\Gamma^{(d)}_\ast$ to this sequence, we conclude that $M^{(\ell)}$ is minimally generated by $\ell+1$ homogeneous elements of the same degree, and that for $1 \leq \ell \leq d-1$, the first syzygy module of $M^{(\ell)}$ is $(M^{(d-1)}(-1))^{\ell}$. Iterating this remark gives a linear minimal free resolution for $M^{(\ell)}$ over $B$.
\end{discussionbox}

\section{Pure resolutions}

\begin{definitionbox}
We say that a finite length $B$-module $M$ has a \define{pure resolution}
if there is a minimal exact sequence of the form
\[
0 \to E_2 \to F_1 \to F_0 \to M \to 0,
\]
where each $F_i$ is generated in a single degree $d_i$, the modules $F_0,
F_1$ are free, and $E_2 = M^{(\ell)}(-d_2)^{\oplus r}$ for some $\ell$ and
$r$. In this case, we call $(d_0, d_1, d_2; \ell)$ the degree sequence of
$M$.

We remark that $\ell=0$ means that the module has finite projective dimension.
\end{definitionbox}

\begin{proposition}
\label{proposition:pure}
If $M$ has a pure resolution of type $(d_0, d_1, d_2; \ell)$, then its
Betti numbers are determined up to scalar multiple. In particular, they are
determined by the first $3$ Betti numbers $(\beta_0, \beta_1, \beta_2)$,
which is a multiple of
\[
\beta^B(d_0,d_1,d_2;\ell) = (d(d_2-d_1) - \ell, d(d_2-d_0)-\ell,
d(d_1-d_0)(\ell+1)).
\]
The other Betti numbers satisfy 
\[
\beta_i = (d-1)^{i-3} \beta_2 \frac{d\ell}{\ell+1}, \quad (i \ge 3).
\]
\end{proposition}

\begin{proof}
The Hilbert series of $B$ is $H_B(t) = \frac{1+(d-1)t}{(1-t)^2}$. Suppose
that $M$ is a finite length module with pure resolution of type $(d_0, d_1,
d_2; \ell)$. By definition, we have an exact sequence of the form
\[
0 \to M^{(\ell)}(-d_2)^{\beta_2} \to B(-d_1)^{\beta_1} \to B(-d_0)^{\beta_0} \to M \to 0
\]
for some $(\beta_0, \beta_1, \beta_2)$. By Discussion~\ref{discussionbox:mcmRNC}, $M^{(\ell)}$ has a resolution of the form
\[
\cdots \to B(-3)^{d(d-1)^2\ell} \to B(-2)^{d(d-1)\ell} \to B(-1)^{d\ell} \to B^{\ell+1} \to M^{(\ell)} \to 0,
\]
so $M$ has a free resolution of the form
\[
\cdots \to B(-d_4)^{\beta_4} \to B(-d_2-1)^{\beta_2 d \ell / (\ell+1)} \to
B(-d_2)^{\beta_2} \to B(-d_1)^{\beta_1} \to B(-d_0)^{\beta_0}.
\]
For $i>3$, we have $d_i = d_{i-1}+1 = d_2 + (i-2)$ and $\beta_i =
(d-1)\beta_{i-1} = (d-1)^{i-3} \beta_2 d \ell / (\ell+1)$. Taking the
alternating sum, we get
\begin{align*}
H_M(t) &= \beta_0 t^{d_0} H_B(t) - \beta_1 t^{d_1} H_B(t) + \beta_2 t^{d_2}
H_B(t) + \beta_2 \frac{d\ell}{\ell+1} t^{d_2} H_B(t) \sum_{i \ge 3} (-1)^i
(d-1)^{i-3} t^{i-2}\\
&= \frac{(\beta_0 t^{d_0} - \beta_1 t^{d_1} + \beta_2
t^{d_2})(1+(d-1)t)}{(1-t)^2} - \beta_2 \frac{d\ell}{\ell+1} t^{d_2+1}
\frac{1 + (d-1)t}{(1-t)^2} \frac{1}{1 - (1-d)t}\\
&= \frac{(\beta_0 t^{d_0} - \beta_1 t^{d_1} + \beta_2 t^{d_2})(1+(d-1)t) -
\frac{d\ell}{\ell+1} \beta_2 t^{d_2+1}}{(1-t)^2}.
\end{align*}
Since $H_M(t)$ is a polynomial, the numerator $h(t)$ of the last expression
is divisible by $(1-t)^2$. This translates to $h(1) = h'(1) = 0$ (where
$h'$ is the derivative with respect to $t$), which gives two linearly
independent conditions on $(\beta_0, \beta_1, \beta_2)$ since $d_0 \ne d_1$ and $d \ne 0$:
\begin{align*}
\begin{pmatrix} d & -d & \frac{d}{\ell + 1} \\
d_0 + (d_0+1)(d-1) & - d_1 - (d_1+1)(d-1) & d_2 + (d_2+1)(\frac{d}{\ell+1} -1 )\end{pmatrix} 
\begin{pmatrix} \beta_0 \\ \beta_1 \\ \beta_2 \end{pmatrix} = 0.
\end{align*}
So $(\beta_0, \beta_1, \beta_2)$ is determined up
to simultaneous scalar multiple, and it is straightforward to check that
$\beta^B(d_0,d_1,d_2;\ell)$ is a valid solution.
\end{proof}

Since it will be used later, we record a relation amongst these pure Betti
tables
\begin{align} \label{eqn:pureredundancy}
\beta^B(d_0, d_1, d_2; \ell) = \left(1-\frac{\ell}{d-1}\right) \beta^B(d_0,
d_1, d_2; 0) + \frac{\ell}{d-1} \beta^B(d_0, d_1, d_2; d-1).
\end{align}
This relation extends to all of the Betti numbers since the later Betti
numbers are multiples of $\beta_2$.

\section{Main result}
\label{section:mainresult}

\begin{theorem} \label{thm:betticoneB}
The extremal rays of the subcone of $\BettiCone(B)$ generated by the Betti
tables of finite length modules are spanned by Betti tables of modules with pure resolutions of type $(d_0, d_1, d_2; \ell)$ where $d_0 < d_1 < d_2$ and $\ell = 0$ or $\ell = d-1$.
\end{theorem}

The proof will be given at the end of the section. The idea is to embed
this cone as a certain quotient cone of $\BettiCone(S)$ and to deduce the
result from \cite{es:bs}.

Let $M$ be a finite length $B$-module. Let $(F_\bullet, \partial_\bullet)$ be a minimal graded $B$-free resolution of $M$; then $F_i = \bigoplus_j B(-j)^{\beta_{i,j}^B(M)}$. Consider the exact sequences 
\begin{align*}
0 \to \image \partial_2 \to F_1 \to \image \partial_1 \to 0, \qquad 0 \to \image \partial_1 \to F_0 \to M \to 0.
\end{align*}
Using \cite[Corollary 18.6]{eisenbud}, we conclude that $\depth(\image \partial_i) = i$ for $i=1,2$, so $\image \partial_2$ is a maximal Cohen--Macaulay $B$-module. By Discussion~\ref{discussionbox:mcmRNC}, we may write
\begin{align*}
\image \partial_2 & = \bigoplus_{\ell=0}^{d-1} \bigoplus_{j \in \ints}
(M^{(\ell)}(-j))^{b_{\ell,j}(M)}, 
\end{align*}
for some integers $b_{\ell,j}(M)$. Hence
\begin{align} \label{eqn:imagepartial3}
\image \partial_3 & =  \bigoplus_{j \in \ints}
(M^{(d-1)}(-j-1))^{s_j} \text{ where } s_j = \sum_{\ell=0}^{d-1}\ell
b_{\ell,j}(M).
\end{align}
Sheafifying the complex $0 \to \image \partial_2 \to F_1 \to F_0$, we
get the locally free resolution 
\[
0 \to 
\bigoplus_{\ell=0}^{d-1} \bigoplus_{j \in \ints}
\strSh(-jd+\ell)^{b_{\ell,j}(M)} \to 
\bigoplus_{j \in \ints}\strSh(-jd)^{\beta_{1,j}^B(M)} \to 
\bigoplus_{j \in \ints}\strSh(-jd)^{\beta_{0,j}^B(M)} 
\]
of $\widetilde{M} = 0$ over $\projective^1$. Applying $\Gamma_*$ to this
complex, we get the complex
\[
0 \to 
\bigoplus_{\ell=0}^{d-1} \bigoplus_{j \in \ints}
S(-jd+\ell)^{b_{\ell,j}(M)} \to 
\bigoplus_{j \in \ints}S(-jd)^{\beta_{1,j}^B(M)} \to 
\bigoplus_{j \in \ints}S(-jd)^{\beta_{0,j}^B(M)},
\]
which is acyclic by \cite[Lemma 20.11]{eisenbud}, and hence a resolution of
an $S$-module, which we denote by $M'$. This resolution is minimal, and
$M'$ is a finite length module. It follows that
\begin{equation}
\label{equation:bettiNosOverS}
\beta_{i,j}^S(M') = 
\begin{cases}
\beta_{i, j/d}^B(M),  & \text{if}\; i\in \{0,1\}
\;\text{and}\; d \mid j, \\
b_{d\lceil j/d\rceil-j, \lceil j/d \rceil}(M), & 
\text{if}\; i=2, \\
0, & \text{otherwise}.
\end{cases}
\end{equation}
Note, parenthetically, that the association $M \mapsto M'$ is functorial.

Since $M^{(\ell)}$ is minimally generated as a $B$-module by $\ell+1$ elements,
we get relations
\begin{equation} \label{eqn:Srelation1}
\begin{split}
\beta_{2,j}^B(M) & = 
\sum_{\ell=0}^{d-1} (\ell+1)b_{\ell,j}(M) = 
\sum_{\ell=0}^{d-1} (\ell+1)\beta^S_{2,jd-\ell}(M'), \;\text{and},\\
\beta_{3,j+1}^B(M) & = 
d\sum_{\ell=0}^{d-1} \ell b_{\ell,j}(M) = 
d\sum_{\ell=0}^{d-1} \ell\beta^S_{2,jd-\ell}(M').
\end{split}
\end{equation} From these, we obtain another relation
\begin{align} \label{eqn:Srelation2}
d \beta_{2,j}^B(M) - \beta_{3,j+1}^B(M) & = 
d\sum_{\ell=0}^{d-1} \beta^S_{2,jd-\ell}(M').
\end{align}

We want to say that the correspondence $M \mapsto M'$
descends to a combinatorial map on Betti tables $\beta^B(M) \mapsto
\beta^S(M')$. Unfortunately, $\beta^B(M)$ does not uniquely determine $\beta^S(M')$ as Example~\ref{eg:crude} shows (one needs the finer invariants $b_{\ell,j}(M)$), so such a map does not exist. 

\begin{example} \label{eg:crude}
Consider the case $d=5$ and the degree sequences $(0,5,i)$ for $i=6,\dots,10$ over the polynomial ring $S = \Bbbk[x,y]$. The respective pure Betti diagrams are
\footnotesize
\begin{multicols}{5}
\begin{Verbatim}[samepage=true]
       0 1 2
total: 1 6 5
    0: 1 . .
    1: . . .
    2: . . .
    3: . . .
    4: . 6 5





\end{Verbatim}
\begin{Verbatim}[samepage=true]
       0 1 2
total: 2 7 5
    0: 2 . .
    1: . . .
    2: . . .
    3: . . .
    4: . 7 .
    5: . . 5




\end{Verbatim}
\begin{Verbatim}[samepage=true]
       0 1 2
total: 3 8 5
    0: 3 . .
    1: . . .
    2: . . .
    3: . . .
    4: . 8 .
    5: . . .
    6: . . 5



\end{Verbatim}
\begin{Verbatim}[samepage=true]
       0 1 2
total: 4 9 5
    0: 4 . .
    1: . . .
    2: . . .
    3: . . .
    4: . 9 .
    5: . . .
    6: . . .
    7: . . 5


\end{Verbatim}
\begin{Verbatim}[samepage=true]
       0 1 2
total: 1 2 1
    0: 1 . .
    1: . . .
    2: . . .
    3: . . .
    4: . 2 .
    5: . . .
    6: . . .
    7: . . .
    8: . . 1
\end{Verbatim}
\end{multicols}
\normalsize
Pick rational numbers $c_1, \dots, c_5$. Then there is some integer $D > 0$ so that the weighted sum of these Betti diagrams with coefficients $Dc_i$ is the Betti table of some finite length $S$-module $N$. We will see in the proof of Lemma~\ref{lem:bsRNC} that $N = M'$ for some $B$-module $M$. The data $(\beta^B_{i,j}(M))_{i=0,1,2,3}$ only contains $4$ numbers which we can express as linear combinations of the $c_i$:
\begin{align*}
\beta^B_{0,0}(M) &= c_1 + 2c_2 + 3c_3 + 4c_4 + c_5\\
\beta^B_{1,1}(M) &= 6c_1 + 7c_2 + 8c_3 + 9c_4 + 2c_5\\
\beta^B_{2,2}(M) &= 5 \cdot 5 c_1 + 4 \cdot 5 c_2 + 3 \cdot 5c_3 + 2 \cdot 5 c_4 + c_5\\
\beta^B_{3,3}(M) &= 5(4 \cdot 5 c_1 + 3 \cdot 5c_2 + 2 \cdot 5 c_3 + 5c_4).
\end{align*}
In particular, for any such data, there are infinitely many $5$-tuples $(c_1, \dots, c_5)$ which give rise to this data, so $(c_1, \dots, c_5)$ cannot be recovered from $\beta^B_{i,j}(M)$ (even up to scalar multiple). \qed
\end{example}

There is an easy solution though: we can define an equivalence relation on
$\BettiCone(S)$ to account for the fact that the sums on the right hand
sides of \eqref{eqn:Srelation1} and \eqref{eqn:Srelation2} are uniquely
determined by $\beta^B(M)$. Then $\beta^S(M')$, under this equivalence
relation, is well-defined since the equivalence relation captures all
possible choices for the $b_{\ell,j}(M)$. We record this discussion now.

\begin{notation}
\label{notation:defnphi}
Define an equivalence relation on $\bbV_S$ and $\BettiCone(S)$ by $\gamma \sim \gamma'$ if
\[
\sum_{\ell=0}^{d-1} \gamma_{2,jd-\ell} = \sum_{\ell=0}^{d-1}
\gamma'_{2,jd-\ell} \text{ and }
\sum_{\ell=0}^{d-1} \ell\gamma_{2,jd-\ell} = \sum_{\ell=0}^{d-1}
\ell\gamma'_{2,jd-\ell} \text{ for all $j$}.
\]
Write ${\BettiCone(S)/\!\sim}$ for the set of equivalence classes under this relation. Let $\phi \colon \bbV_B \to {\bbV_S/\!\sim}$ be the following map: for $\beta \in \bbV_B$, define $\phi(\beta)$ to be the class of any $\gamma \in \bbV_S$ where $\gamma$ 
is such that
\begin{enumerate}
\item \label{enum:defnphiZeroOneNonzero}
$\gamma_{i,j} = \beta_{i,j/d}$ if $i\in \{0,1\}$ and 
$d \mid j$.
\item \label{enum:defnphiTwo}
$\sum_{\ell=0}^{d-1} (\ell+1)\gamma_{2,jd-\ell} = \beta_{2,j}$ 
and $\sum_{\ell=0}^{d-1} \ell\gamma_{2,jd-\ell} = \frac{1}{d}\beta_{3,j+1}$ 
for all $j$.
\item \label{enum:defnphiZeroOnezero}
$\gamma_{i,j} = 0$ if $i\in \{0,1\}$ and $d \nmid j$ or if $i \geq
3$. \hfill \qed
\end{enumerate}
\end{notation}

\begin{lemma} \label{lem:bsRNC}
\begin{enumerate}[\rm (a)]
\item $\phi(\beta^B(M)) \sim \beta^S(M')$.  
\item $\phi(\beta + \beta') \sim \phi(\beta) + \phi(\beta')$.
\item If $\gamma \sim \gamma'$ and $\delta \sim \delta'$, then $\gamma+\delta \sim \gamma'+\delta'$.
\item $\phi(\BettiCone(B)) \subseteq \BettiCone(S)/\sim$.
\item The restriction of $\phi$ to $\BettiCone(B)$ is injective, and its image is generated by the classes of the Betti tables over $S$ of degree sequences of the form $(da_0 < da_1 < a_2)$ where $a_2 \equiv 0, 1 \pmod d$.
\end{enumerate}
\end{lemma}

\begin{proof}
Properties (a), (b), and (c) follow directly from the definition of $\sim$. Since $\BettiCone(B)$ is additively generated by elements of the form $\beta^B(M)$, (d) follows from (a), (b), and (c).

Let $\beta, \beta' \in \BettiCone(B)$. Set $\gamma = \phi(\beta)$, $\gamma'
= \phi(\beta')$.
If $\gamma \sim \gamma'$, then $\beta_{i,j} = \beta'_{i,j}$ 
for all $0 \leq i \leq 3$ and for all $j$. To show that $\phi$ is
injective, we need that for any graded $B$-module $M$, 
$\left(\beta^B_{i,j}(M)\right)_{\substack{0 \leq i\leq3, \\ j \in\ints}}$ 
determines $\beta^B(M)$. Even stronger, by \eqref{eqn:imagepartial3} and \eqref{eqn:Srelation1}, these invariants determine $\image \partial_3$:
\[
\image \partial_3 \cong \bigoplus_{j \in \ints} (M^{(d-1)}(-j))^{\beta^B_{3,j}(M)/d}.
\]

Now we describe the image of $\phi$. Let $a_0, a_1, a_2$ be
integers such that $da_0 < da_1 < a_2$. Let $N$ be a finite length
graded $S$-module with pure resolution with degree sequence
$(da_0 < da_1 < a_2)$. Let $M = \bigoplus_{n \in \ints} N_{dn}$. Then $M$ is a
finite length graded $B$-module. Take a minimal
$S$-free resolution 
\[
0 \to S(-a_2)^{\beta_{2,a_2}^S(N)} \to 
S(-da_1)^{\beta_{1,da_1}^S(N)} \to S(-da_0)^{\beta_{0,da_0}^S(N)}
\]
of $N$. Restricting this complex to degrees $nd$ for $n \in \ints$, we see that
\begin{align*}
b_{\ell,j}(M) &= 
\begin{cases}
\beta_{2,a_2}^S(N), & \text{if}\; jd-\ell = a_2 \;\text{with}\; 0 \leq \ell \leq d-1, \\
0, & \text{otherwise}.
\end{cases}
\end{align*}
and that, for $i=0,1$, $\beta_{i,j}^B(M)  = \beta_{i, jd}^S(N)$.  Note that
$N = M'$, so the class of $\beta^S(N)$ is in $\image \phi$. The converse
inclusion, that $\image \phi$ is inside the cone generated by the classes
of the Betti tables over $S$ of degree sequences of the form
$(da_0 < da_1 < a_2)$ follows from noting that for all $B$-modules $M$,
$\beta^S(M')$ has a decomposition into pure Betti tables of this form
\cite[\S 1]{es:bs}.

We may further impose that $a_2 \equiv 0 \pmod d$ or $a_2 \equiv 1 \pmod d$ if we just want generators for the cone. This follows from what we have just shown, additivity of $\phi$, and the relation \eqref{eqn:pureredundancy}.
\end{proof}

\begin{proof}[Proof of Theorem~\ref{thm:betticoneB}]
Lemma~\ref{lem:bsRNC} shows that the subcone of $\BettiCone(B)$ generated
by Betti tables of finite length $B$-modules is already generated by pure
Betti tables of type $(d_0, d_1, d_2; \ell)$ where $d_0<d_1<d_2$ and $\ell
\in \{0,d-1\}$, and also shows that there exist finite length modules which
have these Betti tables. To show that these are extremal rays of this
subcone, we have to show that no such pure Betti table is a nonnegative
linear combination of the other ones. We know that in $\BettiCone(S)$, the pure Betti tables for different degree sequences have this property. Hence we reduce to fixing $d_0, d_1, d_2$ and showing there are no dependencies as we vary $\ell$. But we only allow $\ell=0$ and $\ell=d-1$, and it is clear that the images of their Betti tables under $\phi$ are not scalar multiples of each other. 
\end{proof}

\begin{remarkbox} \label{rmk:simplicial}
By Theorem~\ref{thm:betticoneB}, the extremal rays of $\BettiCone(B)$ are of the form $(d_0, d_1, d_2; \ell)$ where $\ell=0$ or $\ell=d-1$. The proof also gives a natural correspondence between these extremal rays and a subset of the extremal rays of $\BettiCone(S)$ via 
\[
(d_0,d_1,d_2;0) \leftrightarrow (dd_0, dd_1, dd_2), \qquad (d_0,d_1,d_2;d-1) \leftrightarrow (dd_0,dd_1, dd_2-(d-1)).
\]
The extremal rays in $\BettiCone(S)$ have a partial order structure by pointwise comparison, i.e., $(e_0,e_1,e_2) \le (e'_0,e'_1,e'_2)$ if and only if $e_i \le e'_i$ for $i=0,1,2$. We can transfer this partial order structure to the extremal rays of $\BettiCone(B)$ which gives $(d_0,d_1,d_2;\ell) \le (d'_0, d'_1, d'_2;\ell')$ if and only if $d_0 \le d'_0$, $d_1 \le d'_1$ and $dd_2 - \ell \le dd'_2 - \ell'$. 

We can define a simplicial structure on $\BettiCone(B)$ by defining a simplex to be the convex hull of any set of extremal rays that form a chain in this partial order. Then any two simplices intersect in a common simplex since the same property is true in $\BettiCone(S)$ \cite[Proposition 2.9]{boijsoderberg}. Furthermore, every point $\beta \in \BettiCone(B)$ lies in one of these simplices: from the proof of Lemma~\ref{lem:bsRNC}, we see that $\phi(\beta)$ is a positive linear combination of pure Betti tables corresponding to a chain $\{(da^{(i)}_0, da^{(i)}_1, da^{(i)}_2 - \ell^{(i)})\}$, and using \eqref{eqn:pureredundancy}, we can also assume that it is a chain where $\ell^{(i)} \in \{0,d-1\}$ for all $i$. This allows us to use a greedy algorithm as in \cite[\S 1]{es:bs} to decompose elements of $\BettiCone(B)$ as a positive linear combination of pure diagrams.
\end{remarkbox}

\begin{remarkbox}
\label{remarkbox:depthOne}
We can modify Theorem~\ref{thm:betticoneB} to describe the cone of
Cohen--Macaulay $B$-modules of a fixed depth. We have just described the
depth 0 case, and the depth 2 case corresponds to maximal Cohen--Macaulay
modules, which are easily classified
(Discussion~\ref{discussionbox:mcmRNC}), so the only interesting case
remaining is depth 1. In this case, one sheafifies the complex $0 \to \image
\partial_1 \to F_0$ and the resulting module $M'$ is Cohen--Macaulay of
depth 1 (it has a length 1 resolution, and its Hilbert polynomial is the
same as the Hilbert polynomial of $M$, and hence has dimension 1). The
equivalence relation $\sim$ on $\BettiCone(S)$ needs to be changed, but the
required changes are straightforward. The end result is that we can define
depth 1 Cohen--Macaulay modules with pure resolutions (their type is of the
form $(d_0, d_1; \ell)$) and the analogue of Theorem~\ref{thm:betticoneB}
holds.
\end{remarkbox}

\section{An example} \label{sec:example}

We give a few explicit examples for $d=3$. In this case, $B$ is the homogeneous coordinate ring of the rational normal cubic. We will use {\tt Macaulay2} \cite{M2} and the package {\tt BoijSoederberg}.

We wish to construct a finite length $B$-module with pure resolution of type $(d_0, d_1, d_2; \ell)$ where $0 \le \ell \le 2$. Consider the case $(0, 2, 3; 1)$. Let $N$ be a finite length module over $S = \Bbbk[x,y]$ with pure resolution of degree sequence $0 < 6 < 8$, for example we can take $N$ to be the quotient by the ideal of $4$ random sextics. In any case we have $N = \bigoplus_{i=0}^6 N_i$ and we set $M = N_0 \oplus N_3 \oplus N_6$, which is a $B$-module. If we consider the free resolution $0 \to S(-8)^3 \to S(-6)^4 \to S$ for $N$ and throw out all graded pieces whose degree is not divisible by $3$ (and then divide all remaining degrees by $3$), then we get the exact sequence
\[
0 \to M^{(1)}(-3)^3 \to B(-2)^4 \to B \to M \to 0.
\]

We now give an example of decomposing the Betti table of a $B$-module $M$. 
Set $a = x^3$, $b = x^2y$, $c = xy^2$, $d=y^3$ so that we can identify $B$ as the polynomial ring in $a,b,c,d$ modulo the $2 \times 2$ minors of $\begin{pmatrix} a & b & c \\ b & c & d\end{pmatrix}$. Consider the $B$-module $M = B/I$ where $I$ is the ideal $(a+c, d^2, cd)$. The Betti table of $M$ over $B$ is
\footnotesize 
\begin{Verbatim}[samepage=true]
       0 1 2 3  4  5 
total: 1 3 5 9 18 36 
    0: 1 1 . .  .  . 
    1: . 2 5 9 18 36 ...
\end{Verbatim} 
\normalsize
and we wish to decompose it as a nonnegative sum of pure diagrams.
Define an $S$-module $M'$ by using the same presentation matrix. Then $M' = S/J$ where $J$ is the ideal $(x^3+xy^2, y^6, xy^5)$. Its Betti table and its decomposition into a nonnegative sum of pure Betti tables is:
\footnotesize
\begin{Verbatim}[samepage=true]
       0 1 2    1 /       0 1 2\     2 /       0 1 2\    1 /       0 1 2\
total: 1 3 2   (-)|total: 4 7 3| + (--)|total: 1 7 6| + (-)|total: 1 4 3| 
    0: 1 . .    7 |    0: 4 . .|    21 |    0: 1 . .|    3 |    0: 1 . .|
    1: . . .      |    1: . . .|       |    1: . . .|      |    1: . . .|
    2: . 1 . =    |    2: . 7 .|       |    2: . . .|      |    2: . . .|
    3: . . .      |    3: . . .|       |    3: . . .|      |    3: . . .|    
    4: . . .      |    4: . . .|       |    4: . . .|      |    4: . . .|
    5: . 2 1      \    5: . . 3/       \    5: . 7 6/      |    5: . 4 .|
    6: . . 1                                               \    6: . . 3/
\end{Verbatim} 
\normalsize
These $3$ pure diagrams translate to the exact sequences
\small \begin{align*}
0 \to M^{(2)}(-3)^3 \to B(-1)^7 \to B^4, \quad 0 \to M^{(2)}(-3)^6 \to B(-2)^7 \to B, \quad 0 \to M^{(1)}(-3)^6 \to B(-2)^4 \to B,
\end{align*}
\normalsize
and hence we get the sum of pure diagrams
\footnotesize
\begin{Verbatim}[samepage=true]
 1 /       0 1 2  3  4  5    \     2 /       0 1  2  3  4   5    \    1 /       0 1 2 3  4  5    \
(-)|total: 4 7 9 18 36 72    | + (--)|total: 1 7 18 36 72 144    | + (-)|total: 1 4 6 9 18 36    |
 7 |    0: 4 7 .  .  .  .    |    21 |    0: 1 .  .  .  .   .    |    3 |    0: 1 . . .  .  .    |
   \    1: . . 9 18 36 72 .../       \    1: . 7 18 36 72 144 .../      \    1: . 4 6 9 18 36 .../ 
\end{Verbatim} 
\normalsize

Alternatively, we can use Remark~\ref{rmk:simplicial} to get a decomposition of $\beta^B(M)$ without understanding $\beta^S(M)$. Then the greedy algorithm in \cite[\S 1]{es:bs} tells us to subtract the largest positive multiple of the pure diagram of type $(0,1,3;2)$ that leaves a nonnegative table. By Proposition~\ref{proposition:pure}, this has Betti table
\footnotesize 
\begin{Verbatim}[samepage=true]
       0 1 2  3  4  5 
total: 4 7 9 18 36 72
    0: 4 7 .  .  .  . 
    1: . . 9 18 36 72 ...
\end{Verbatim}
\normalsize
So the largest multiple we can subtract is $1/7$, which leaves us with
\footnotesize 
\begin{Verbatim}[samepage=true]
  1 /       0  1  2  3  4   5    \
(--)|total: 4 14 26 45 90 180    |
  7 |    0: 3  .  .  .  .   .    |
    \    1: . 14 26 45 90 180 .../
\end{Verbatim}
\normalsize
Now we repeat by subtracting the largest possible multiple of the pure diagram of type $(0,2,3;2)$ that leaves a nonnegative table. When we do this, the result is another pure diagram. The final decomposition is
\footnotesize
\begin{Verbatim}[samepage=true]
 1 /       0 1 2  3  4  5    \     5 /       0 1  2  3  4   5    \    1 /       0 1 2 \
(-)|total: 4 7 9 18 36 72    | + (--)|total: 1 7 18 36 72 144    | + (-)|total: 1 3 2 |
 7 |    0: 4 7 .  .  .  .    |    28 |    0: 1 .  .  .  .   .    |    4 |    0: 1 . . |
   \    1: . . 9 18 36 72 .../       \    1: . 7 18 36 72 144 .../      \    1: . 3 2 / 
\end{Verbatim} 
\normalsize

Using \eqref{eqn:pureredundancy}, this pure diagram decomposition of $\beta(M)$ is equivalent to the previous one.

\section{Questions}

\begin{enumerate}[1.]
\item \label{question:allfgmodules}
Unfortunately, our techniques do not allow us to describe the cone
$\BettiCone(B)$ of all finitely generated $B$-modules (i.e., allowing those
that are not Cohen--Macaulay). Given the situation for polynomial rings
\cite{bs2}, we might conjecture that $\BettiCone(B)$ is the sum (over
$c=0,1,2$) of the cones of Betti tables for Cohen--Macaulay $B$-modules of
codimension $c$. Is this correct?

\item For the polynomial ring, the inequalities that define the facets of
its cone of Betti tables has an interpretation in terms of cohomology
tables of vector bundles on projective space \cite[\S 4]{es:bs}. Are there
interpretations for the inequalities that define the cone of finite length
$B$-modules?
\end{enumerate}

\begin{remarkbox}
\label{remarkbox:totalBetti}
With reference to Question~\ref{question:allfgmodules}, let us look at the
cone $\BettiCone_{\mathrm{tot}}(B)$ 
generated by the total Betti numbers $(b_0(M), b_1(M), b_2(M), b_3(M))
\in \rationals^4$ of finitely generated graded $B$-modules $M$. Consider an
exact sequence 
\[
0 \to E_3 \to F_2 \to F_1 \to F_0 \to M \to 0
\]
such that $F_0$, $F_1$, and $F_2$ are free, $E_3$ is a direct sum of copies
of $M^{(d-1)}$ and $\image(F_{i+1} \to F_i) \subseteq
\frakm F_i$ for $i=0,1$. (See Discussion~\ref{discussionbox:mcmRNC}.)
Note that for $i=0,1,2$, $b_i(M) = \rank F_i$ and that 
$b_3(M) = d \rank(E_3)$. By considering the partial Euler
characteristics of the above exact sequence, we get four inequalities:
\begin{align*}
b_3(M) \geq 0,\  b_2(M) \ge \frac{b_3(M)}{d-1}, \  b_1(M) \ge b_2(M) - \frac{b_3(M)}{d}, \ b_0(M) \ge b_1(M) - b_2(M) + \frac{b_3(M)}{d}.
\end{align*}
To prove the second inequality, we have an exact sequence $0 \to E_3 \to F_2 \to N \to 0$ where $N$ is a maximal Cohen--Macaulay module, and so $\rank(E_3) \le (d-1) \rank(N)$. Consider the set
\[
\{(b_0, b_1, b_2, b_3) \in \rationals^4 : b_3 \geq 0,\ 
b_2 - \frac{b_3}{d-1} \geq 0,\
b_1 - b_2 + \frac{b_3}{d} \geq 0,\ 
b_0 - b_1 + b_2 - \frac{b_3}{d} \ge 0\}.
\]
This is a convex polyhedral 
cone, with extremal rays generated by $(1,0,0,0)$, $(1,1,0,0)$, $(0,1,1,0)$ and $(0,1,d,d(d-1))$.
We claim that this is the closure of
$\BettiCone_{\mathrm{tot}}(B)$; of course, the rays generated by $(0,1,1,0)$ and $(0,1,d,d(d-1))$ do not belong to $\BettiCone_{\mathrm{tot}}(B)$. This picture, and the proof below, are analogous to the case of regular local
rings~\cite[\S 2]{BEKS}. The point $(1,0,0,0)$ comes from a free module
of rank one, while $(1,1,0,0)$ comes from $M = B/(f)$ for some non-zero $f
\in B$. 

Consider the modules $M_t$, $t \ge 1$ with pure resolutions of type $(0, t, t+1; 0)$. By Proposition~\ref{proposition:pure}, $(b_0(M_t), b_1(M_t), b_2(M_t), b_3(M_t))$ is a multiple of $(1, t+1, t, 0)$, which limits to the ray $(0,1,1,0)$ as $t \to \infty$. Now consider modules $N_t$, $t \geq 1$ with pure resolutions of type
$(0, td, td+1; d-1)$. By Proposition~\ref{proposition:pure}, 
$(b_0(N_t), b_1(N_t), b_2(N_t), b_3(N_t))$ is a multiple of $(1, td^2+1, td^3, td^3(d-1))$, which, as $t \to \infty$, approaches the ray generated by $(0,1,d,d(d-1))$.
\end{remarkbox}

\begin{remarkbox}
One might wonder whether a similar argument works for 
the Veronese embedding $(\projective^2, \strSh_{\projective^2}(2))$, whose homogeneous coordinate
ring is the only other Veronese subring with finite Cohen--Macaulay
representation type. There are significant obstacles to overcome, which we outline. In
\S\ref{section:mainresult}, we took the sheafification of a
resolution $0 \to \image \partial_2 \to F_1 \to F_0$ of the finite length
$B$-module $M$ by \emph{maximal Cohen--Macaulay} $B$-modules and,
thereafter, applied $\Gamma_*$ to obtain a minimal $S$-free resolution of
the finite length $S$-module $M'$; the key point is that for a 
maximal Cohen--Macaulay $B$-module $N$, $\Gamma_*(\widetilde{N})$ is a
maximal Cohen--Macaulay (hence free) $S$-module. This is not true for the
Veronese embedding $(\projective^2, \strSh_{\projective^2}(2))$. 

More specifically, set $S = \Bbbk[x,y,z]$ and $B = \bigoplus_n S_{2n}$. Then, up to twists, $B$ has three non-isomorphic maximal Cohen--Macaulay modules 
$M^{(0)} \simeq B$, the canonical module $M^{(1)}$ and the syzygy module
$M^{(3)}$ of $M^{(1)}$ (see the proof of \cite[Proposition 16.10]{yoshino}). The first syzygy of $M^{(\ell)}$ is $(M^{(3)})^{\oplus \ell}$, for $\ell = 0, 1, 3$. However, $\Gamma_*(\widetilde{M^{(3)}})$ is not maximal 
Cohen--Macaulay over $S$; its depth is two. To see this, note that the
exact sequence $0 \to M^{(3)} \to B^3 \to M^{(1)} \to 0$ gives the Euler
sequence $0 \to \Omega_{\projective^2}^1(1) \to 
\strSh_{\projective^2}^3 \to \strSh_{\projective^2}(1) \to 0$
on $\projective^2$; it follows that 
$\Gamma_*(\widetilde{M^{(3)}})$ is the second syzygy of $\Bbbk(1)$ as an
$S$-module and has depth two. From this it follows that if we begin with a
$B$-free resolution $(F_\bullet, \partial_\bullet)$ of a $B$-module of
finite length and apply $\Gamma_*$ to the sheafification of
$0 \to \image \partial_3 \to F_2 \to F_1 \to F_0$, the ensuing
complex of $S$-modules need not consist of free $S$-modules.
\end{remarkbox}

\end{document}